\documentclass[11pt]{article}
\usepackage{amssymb}
\usepackage{color}
\usepackage{soul}

\newcommand{\bP}{\mathbb{P}}
\newcommand{\N}{\mathbb{N}}
\newcommand{\ord}{\mbox{\rm ord }}
\newcommand{\ini}{\mbox{\rm in}}

\newcommand{\cone}{\mathrm{cone}\,}

\newlength{\szer}

\newtheorem{defi}{Definition}[section]

\newtheorem{ejemplos}[defi]{Examples}

\newtheorem{teorema}[defi]{Theorem}
\newtheorem{prop}[defi]{Proposition}
\newtheorem{lema}[defi]{Lemma}
\newtheorem{coro}[defi]{Corollary}

\newenvironment{proof}[1][Proof]{\textbf{#1.} }{\
\rule{0.5em}{0.5em}}

\begin{document}
\title{Contact exponent and the Milnor number of plane curve singularities
\footnotetext{
     \noindent   \begin{minipage}[t]{4in}
       {\small
       2010 {\it Mathematics Subject Classification:\/} Primary 32S05;
       Secondary 14H20.\\
       Key words and phrases: {contact exponent, logarithmic distance, Milnor number, semigroup associated with a branch.\\}
       The first-named author was partially supported by the Spanish Project
    MTM 2016-80659-P.}
       \end{minipage}}}

\author{Evelia R.\ Garc\'{\i}a Barroso and Arkadiusz P\l oski}

\maketitle

\begin{abstract}
\noindent We investigate properties of the contact exponent (in the 
sense of Hironaka \cite{Hironaka}) of plane algebroid
curve singularities over algebraically closed fields of arbitrary characteristic. 
We prove that the contact exponent is an equisingularity invariant and give a new proof of the stability of the maximal contact. Then we prove a bound for the Milnor number and determine the equisingularity class of algebroid curves for which this bound is attained. We do not use the method of Newton's diagrams. Our tool is the logarithmic distance  developed in  \cite{GB-P1}.
\end{abstract}

\section*{Introduction}
Let $C$ be a plane algebroid curve of multiplicity $m(C)$ defined over an algebraically closed field $K$. To calculate the number of infinity near $m(C)$-fold points, Hironaka \cite{Hironaka} (see also \cite{Brieskorn-Knorrer} or \cite{Teissier}) introduced the concept of contact exponent $d(C)$ and study its properties using Newton's diagrams.\\

In this note we prove an explicit formula for a generalization of contact exponent (Section \ref{sect2}, Theorem \ref{principal}) using the logarithmic distance on the set of branches Then we give a new proof of the stability of maximal contact (Section \ref{sect3}, Theorem \ref{th:Hironaka}) without resorting to Newton's diagrams. In Section \ref{sect4} we define the Milnor number $\mu(C)$ in the case of arbitrary characteristic (see \cite{Melle-Wall} and \cite{GB-P2}), prove the bound $\mu(C)\geq (d(C)m(C)-1)(m(C)-1)$ and characterize the singularities for which the bound is attained. In Section \ref{sect5} we reprove the formulae for the contact exponents of higher order (see \cite{LJ} and \cite{Campillo-libro}). Section \ref{sect6} is devoted to the relation between polar invariants and the contact exponent in characteristic zero.

\section{Preliminaries}

Let $K[[x,y]]$ be the ring of formal power series with coefficients in  an algebraically closed field $K$ of arbitrary characteristic. For any non-zero power series $f=f(x,y)=\sum_{{i,j}} c_{ij}x^iy^{j}\in K[[x,y]]$ we define its {\em order} as $\ord f=\inf \{i+j\;:\;c_{ij}\neq 0\}$ and its {\em initial form} as $\ini f=\sum_{i+j=n} c_{ij}x^iy^{j}$, where $n=\ord f$. We let $(f,g)_{0}=\mathrm{dim}_{K}K[[x,y]]/(f,g),$ and called the {\em intersection number} of $f$ and $g$, where $(f,g)$ denotes the ideal of $K[[x,y]]$ generated by $f$ and $g$.\\ 

Let $f$ be a nonzero power series without constant term. An {\em algebroid curve} $C:\{f=0\}$ is defined to be the ideal generated by $f$ in $K[[x,y]]$. The {\em multiplicity} of $C$ is $m(C)=\ord f$. Let $ \bP^{1}(K)$ denotes  the projective line over $K$. The {\em tangent cone} of $C$ is by definition $\cone(C)=\{(a:b)\in \bP^{1}(K)\;:\;\ini f(a,b)=0\}$.\\

The curve $C:\{f=0\}$ is {\em reduced} (resp. {\em irreducible}) if the power series $f$ has no multiple factors (resp. is irreducible). Irreducible curves are called  branches. If  $\sharp \cone (C)=1$  then the curve $C:\{f=0\}$  is called {\em unitangent}. Any irreducible curve is unitangent. For $C:\{f=0\}$ and $D:\{g=0\}$ we put $(C,D)_{0}=(f,g)_{0}$. Then $(C,D)_{0}\geq m(C)m(D),$  with equality if and only if their cones are disjoint.\\

For any sequence $C_{i}:\{f_{i}=0\,:\,1\leq i\leq k\}$ of curves we put $C=\bigcup_{i=1}^{k}C_{i}:\{f_{1}\cdots f_{k}=0\}$. If $C_{i}$ are irreducible and $C_{i}\neq C_{j}$ for $i\neq j$ then we call $C_{i}$ the {\em irreducible components} of $C$.\\

Consider an irreducible power series $f\in K[[x,y]]$. The set 

\[\Gamma(C)=\Gamma(f):=\{(f,g)_{0}\;:\; g\in K[[x,y]],\;\; g\not\equiv 0 \;\hbox{\rm (mod $f$)}  \}\]

is the {\em semigroup} associated with $C:\{f=0\}$.  Note that $\min(\Gamma(C)\backslash \{0\})=m(C)$. It is well-known that $\gcd(\Gamma(C))=1$.\\

The branch $C$ is {\em smooth} (that is its multiplicity equals 1) if and only if $\Gamma(C)=\N$.\\

Two branches $C:\{f=0\}$ and $D:\{g=0\}$ are {\em equisingular} if $\Gamma(C)=\Gamma(D)$. 

Two reduced curves $C:\{f=0\}$ and $D:\{g=0\}$ are equisingular if and only if $f$ and $g$ have the same number $r$ of irreducible factors and there is a factorization $f=f_{1}\cdots f_{r}$ and $g=g_{1}\cdots g_{r}$ such that

\begin{enumerate}
\item[\hbox{\rm (1)}] the branches $C_{i}:\{f_{i}=0\}$ and $D_{i}:\{g_{i}=0\}$ are equisingular for $i\in \{1,\ldots, r\}$,

and

\item[\hbox{\rm (2)}] $(C_i,C_j)_{0}=(D_i,D_j)_{0}$ for any $i,j\in \{1,\ldots, r\}$.
\end{enumerate}

A function $C \mapsto I(C)$ defined on the set of all reduced curves is an  {\em equisingularity invariant}  if $I(C)=I(D)$ for equisingular curves $C$ and $D$. Note that the multiplicity $m(C)$, the number of branches $r(C)$ and the number of tangents  $t(C)$  (which is the cardinality of the $\cone(C)$) of the reduced curve $C$ are equisingularity invariants.\\

For any reduced curve $C:\{f=0\}$ we put ${\cal O}_{C}=K[[x,y]]/(f)$ and 
$\overline {{\cal O}}_{C}$ its integral closure. Let ${\cal C}=\overline {{\cal O}}_{C}:{\cal O}_{C}$ be the {\em conductor} of 
$\overline {{\cal O}}_{C}$ in ${\cal O}_{C}$. The number 
$c(C)=\dim_{K}\overline {{\cal O}}_{C}/{\cal C}$ is the {\em degree of the conductor}. If 
$C$ is a branch then $c(C)$ equals to the smallest element of $\Gamma(C)$ such that $c(C)+N\in \Gamma(C)$ for all $N\in \N$.\\

Suppose that $C$ is a branch. Let $(v_{0},v_{1},\ldots,v_{g})$ be the {\em minimal system of generators} of $\Gamma(C)$ defined by the following conditions:

\begin{enumerate}
\item[\hbox{\rm (3)}] $v_{0}=\min(\Gamma(C)\backslash \{0\})=m(C).$

\item[\hbox{\rm (4)}] $v_{k}=\min(\Gamma(C)\backslash \N v_{0}+\cdots +
 \N v_{k-1})$, for $k\in \{1,\ldots, g\}$.
 
 \item[\hbox{\rm (5)}] $\Gamma(C)=\N v_{0}+\cdots +\N v_{g}$.
\end{enumerate}

In what follows we write $\Gamma(C)=\langle v_{0},v_{1},\ldots, v_{g}\rangle$ when  $v_{0}<v_{1}<\cdots< v_{g}$ is the increasing sequence of minimal system of generators of $\Gamma(C)$.\\

Since $\gcd(\Gamma(C))=1$  the sequence $v_{0},\ldots, v_{g}$ is well-defined. Let $e_{k}:=\gcd(v_{0},\ldots,v_{k})$ for $0\leq k \leq g$. We define the {\em Zariski pairs} $(m_{k},n_{k})=\left(\frac{v_{k}}{e_{k}},\frac{e_{k-1}}{e_{k}}\right)$ for $1\leq k \leq g$. One has $c(C)=\sum_{k=1}^{g}\left(n_{k}-1\right)v_{k}-v_{0}+1$ (see\cite[Corollary 3.5]{GB-P1}).\\

If $K$ is a field of charateristic zero the Zariski pairs determine the {\em Puiseux pairs} and vice versa.\\

If $\Gamma(C)=\langle v_{0},v_{1},\ldots, v_{g}\rangle$ then the sequence $(v_{i})_{i}$ is {\em strongly increasing}, that is $n_{i-1}v_{i-1}<v_{i}$ for $i\in\{2,\ldots, g\}$.\\

Let $C:\{f=0\}$ be a reduced unitangent curve of multiplicity $n$. Let us consider two cases:
\begin{enumerate}
\item [\hbox{\rm (i)}]$f=c(y-ax)^{n}+\hbox{\rm higher order terms},$ where $a,c\in K$, $c\neq 0$
and
\item [\hbox{\rm (ii)}]$f=cx^{n}+\hbox{\rm higher order terms}$, $c\in K\backslash \{0\}$.
\end{enumerate}

We associated with $C$ a power series $f_{1}=f_{1}(x_{1},y_{1})\in K[[x_{1},y_{1}]]$ by putting $f_{1}(x_{1},y_{1})=x_{1}^{-n}f(x_{1},ax_{1}+x_{1}y_{1}) $ in the case (i) and $f_{1}(x_{1},y_{1})=y_{1}^{-n}f(x_{1}y_{1},y_{1})$ otherwise. The {\em strict quadratic transform} of $C:\{f=0\}$ is the curve $\widehat C:\{f_{1}=0\}.$\\

Obviously $m(\widehat C)\leq m(C)$. If $C=\bigcup_{i=1}^{k}C_{i}$ is a unitangent curve then $C_{i}$ are unitangent and $\widehat C=\bigcup_{i=1}^{k}\widehat C_{i}$.\\

The following lemma is a particular case of a theorem due to Angerm\"uller \cite[Lemma II.2.1]{Angermuller}.

\begin{lema}
\label{prop:Angermuller}
Let $C$ be a singular branch. Then the strict quadratic transform $\widehat C$ of $C$ is also a plane branch. If $\Gamma(C)=\langle v_{0},\ldots, v_{g} \rangle$  then
\begin{itemize}
\item $\Gamma(\widehat C)=\langle v_{0}, v_{1}-v_{0}, \ldots\rangle$ if $v_{0}<v_{1}-v_{0}$

or

\item $\min (\Gamma(\widehat C)\backslash \{0\})=v_{1}-v_{0}$ if $v_{1}-v_{0}<v_{0}$.
\end{itemize}
\end{lema}

\section{Logarithmic distance}
\label{sect2}
A {\em log-distance} $\delta$ associates with any two branches $C,D$ a number $\delta(C,D)\in \mathbb R_{+}\cup \{+\infty\}$ such that for any branches $C$, $D$ and $E$ we have:\\

\begin{enumerate}
\item[($\delta_1$)] $\delta (C,D)=\infty$ if and only if $C=D$,

\item[($\delta_2$)] $\delta(C,D)=\delta(D,C)$,

\item[($\delta_3$)] $\delta(C,D)\geq \hbox{\rm inf}\{\delta(C,E),\delta(E,D)\}$.\\
\end{enumerate}
Note that  if $\delta(C,E)\neq \delta(E,D)$ then $\delta(C,D)=\hbox{\rm inf}\{\delta(C,E),\delta(E,D)\}$.

\medskip
If $C$ and $D$ are reduced curves with irreducible components $C_{i}$ and $D_{j}$ then we set  $\delta(C,D):=\hbox{\rm inf }_{i,j}\{\delta(C_i,D_j)\}$. 

\medskip
If $\delta$ is  a log-distance then $\Delta:=\frac{1}{\delta}$ (by convention $\frac{1}{+\infty}=0$) is an {\em ultrametric} on the set of branches and vice versa: if $\Delta$ is an ultrametric then $\frac{1}{\Delta}$ is a log-distance.

\begin{ejemplos} $\,$
\label{ej}

\begin{enumerate}
\item The order of contact of branches $d(C,D)=\frac{(C,D)_{0}}{m(C)m(D)}$ is a log-distance (see \cite[Corollary 2.9]{GB-P1}).
\item The minimum number of quadratic transformations $\gamma(C,D)$ necessary to separate $C$ from $D$ is a log-distance (see \cite[Theorem 3]{Waldi}).
\end{enumerate}
\end{ejemplos}

\noindent Let $\delta$ be a log-distance.
\begin{lema}
\label{elemental}
If $C$ has $r>1$ branches $C_i$ 
and $D$ is any branch then $\delta(C,D)\leq \hbox{\rm inf}_{i,j}\{\delta(C_i,C_j)\}$.
\end{lema}
\begin{proof} Let $i_0,j_0$ be such that $\inf_{i,j}
\{\delta(C_i,C_j)\}=\delta(C_{i_0},C_{j_0})$. Then $\delta(C,D)=\inf_{1\leq i
\leq r}\{\delta(C_i,D)\}
\leq \inf\{\delta(C_{i_0},D),\delta(C_{j_0},D)\}$ and using $(\delta3)$ we 
get $\delta(C,D)\leq \delta(C_{i_0},C_{j_0}),$ which proves 
the lemma.
\end{proof}

\medskip

Let $C$ be a reduced curve. For 
every non-empty family of branches ${\cal B}$ we put
\[\delta(C,{\cal B}):=\sup\{\delta(C,W)\;:\;W\in {\cal B}\}.\]

\medskip
Note that $\delta(C,{\cal B})=+\infty$ if $C\in {\cal B}$. 
\medskip
In what follows we assume the following condition\\

(*) $\;\;$ for any branch $C$ there exists  $W_{0}\in {\cal B}$ such that $\delta(C, {\cal B})=\delta(C,W_{0})$,\\
and  we say that $W_{0}$ has {\em maximal $\delta$-contact } with $C$.\\

We will prove the following

\begin{teorema}
\label{principal}
Let $C$ be a reduced curve with 
$r>1$ branches $C_i$ and let ${\cal B}$ be a family of 
branches such that the condition (*) holds.

Then
\[\delta(C,{\cal B})=\inf\{\inf_i \{\delta(C_i,{\cal B})\}, 
\inf_{i,j} \{\delta(C_i,C_j)\}\}.\]

Moreover, there exists  $i_0\in \{1,\ldots,r\}$ such 
that if a branch $W\in {\cal B}$ has maximal $\delta$-contact with $C_{i_0}$ then it has maximal $\delta$-contact with $C$.
\end{teorema}

\begin{proof}
Set $\delta^*(C,{\cal B})=\inf\{\hbox{\rm inf}_i 
\delta(C_i,{\cal B}), \inf_{i,j} \delta(C_i,C_j)\}$.

\medskip

The inequality $\delta(C,{\cal B})\leq \delta^*(C,{\cal B})$ 
follows from  
Lemma \ref{elemental} and from the definition of $\delta(C_i,{\cal B})$. 
Thus to prove the result let us consider two cases:

\medskip
First case: $\inf_i \{\delta(C_i,{\cal B})\} 
\leq  \inf_{i,j} \{\delta(C_i,C_j)\}$.

Let $i_0\in \{1,\ldots,r\}$ be such that $\delta(C_{i_0},{\cal B})=
\inf_i \{\delta(C_i,{\cal B})\}$. Then, we have 
\begin{equation}
\label{uno}
\delta(C_{i_0},{\cal B})=\delta^*(C,{\cal B}).
\end{equation}

Let $W\in {\cal B}$ such that $\delta(C_{i_0},W)=\delta(C_{i_0},
{\cal B})$. We claim that
\begin{equation}
\label{dos}
\delta(C_{i_0},W)\leq \delta(C_i,W) \;\;\hbox{\rm for all } i\in \{1,\ldots,r\}.
\end{equation}
To obtain a contradiction suppose that (\ref{dos}) does not hold.
Thus there is $i_1\in\{1,\ldots,r\}$ such that
\begin{equation}
\label{tres}
\delta(C_{i_1},W) < \delta(C_{i_0},W).
\end{equation}
Applying Property $(\delta_3)$ to the branches $C_{i_0}, C_{i_1}$
 and $W$ we get
\begin{equation}
\label{cuatro}
\delta(C_{i_1},W)=\delta(C_{i_0},C_{i_1}).
\end{equation}
On the other hand, in the case under consideration we have
\begin{equation}
\label{cinco}
\delta(C_{i_0},{\cal B})=\inf_i \{\delta(C_i,{\cal B})\}
\leq  \delta(C_{i_0},C_{i_1}).
\end{equation}
Therefore by (\ref{cinco}), (\ref{cuatro}) and (\ref{tres}) we get
$\delta(C_{i_0},{\cal B})\leq \delta(C_{i_1},W)< \delta(C_{i_0},W),$
which contradicts the definition of $\delta(C_{i_0},{\cal B})$.

\medskip
Now, using (\ref{dos}) and (\ref{uno}), we compute
\[\delta(C,W)=\inf\{\delta(C_{i_0},W),\inf_{i\neq i_0} (\delta(C_i,W))\}
=\delta(C_{i_0},W)=\delta(C_{i_0},{\cal B})=\delta^*(C,{\cal B}),\]

which proves the theorem in the first case.

\medskip
Second case: $\inf_i \{\delta(C_i,{\cal B})\} > 
\inf_{i,j} \{\delta(C_i,C_j)\}$.

Let $i_0,j_0$ be such that $\delta(C_{i_0},C_{j_0})=
\inf_{i, j} \delta(C_i,C_j)=\delta^*(C,{\cal B})$.

Let $W\in {\cal B}$ such that $\delta(C_{i_0},W)=\delta(C_{i_0},
{\cal B})$. We claim that
\begin{equation}
\label{ocho}
\delta(C_{i_0},C_{j_0})\leq \delta(C_i,W)\;\;\hbox{\rm for all } i\in \{1,\ldots,r\} 
\;\;\hbox{\rm with equality for } i=j_0.
\end{equation}
\noindent First observe that in the case under consideration we have
\begin{equation}
\label{nueve}
\delta(C_{i_0},C_{j_0})<\delta(C_{i_0},{\cal B})=\delta(C_{i_0},W).
\end{equation}

Fix $i\in\{1,\ldots,r\}$. If $\delta(C_{i_0},W)\leq \delta(C_i,W)$ 
then (\ref{ocho}) follows from (\ref{nueve}). If $\delta(C_i,W)<\delta(C_{i_0},W)$ 
then by Property $(\delta_3)$ applied to the branches $C_i,C_{i_0}$ 
and $W$ we get $\delta(C_i,W)=\delta(C_i,C_{i_0})\geq \inf_{i,j}\{\delta(C_i,C_j)\}
=\delta(C_{i_0},C_{j_0})$. In particular for $i=j_0$, $\delta(C_{j_0},W)=\delta(C_{j_0},C_{i_0})
=\delta(C_{i_0},C_{j_0})$.

Now, by the definition of $\delta(C,W)$ and inequalities  
(\ref{ocho}) and (\ref{nueve}) we get:
\begin{eqnarray*}
\delta(C,W)&=&\inf\{\delta(C_{i_0},W),\delta(C_{j_0},W),
\inf_{i\neq i_0,j_0}\delta(C_i,W)\}\\
 &=& \delta(C_{j_0},W)=\delta(C_{i_0},C_{j_0})
=\delta^*(C,{\cal B}),
\end{eqnarray*}
 which proves the theorem in the second case.
\end{proof}

\begin{prop}
\label{lema}
Let $C$ and $D$ be two branches. Then
\begin{enumerate}
\item If there exists a branch of ${\cal B}$ which has maximal $\delta$-contact with $C$ 
and $D$ then $\delta(C,D)\geq \inf\{\delta(C,{\cal B}),\delta(D,{\cal B})\}$ with equality if 
$\delta(C,{\cal B})\neq \delta(D,{\cal B})$.
\item If there does not exist such a branch  and $U$ has maximal $\delta$-contact with $C$ and $V$ has maximal $\delta$-contact with $D$ then
$\delta(C,D)=\delta(U,V)<\inf\{\delta(C,{\cal B}),\delta(D,{\cal B})\}.$
\end{enumerate}
\end{prop}

\begin{proof} (see \cite[Proposition 2.2]{GB-L-P} for $\delta=d$).

\noindent If there exists a branch $W\in {\cal B}$ such that $\delta(W,C)=\delta(C,{\cal B})$ and $\delta(W,D)=\delta(D,{\cal B})$ then we get the first part of the proposition by using Proper\-ty $(\delta3)$
to the branches $C$, $D$ and $W$. In order to check the se\-cond part 
suppose that such a branch does not exist. Let $U,V\in {\cal B}$  
such that $\delta(U,C)=\delta(C,{\cal B})$ and $\delta(V,D)=\delta(D,{\cal B})$. By hypothesis $\delta(C,V)<\delta(C,{\cal B})
=\delta(C,U)$ and $\delta(D,U)<\delta(D,{\cal B})=\delta(D,V)$. According to $(\delta3)$ we get $\delta(U,V)=\inf\{\delta(C,V),\delta(C,U)\}=\delta(C,V)$ and $\delta(U,V)=\inf \{\delta(D,U),\delta(D,V)\}=\delta(D,U)$ thus
\begin{equation}
\label{igualdad uno}
\delta(C,V)=\delta(D,U)=\delta(U,V).
\end{equation} 

\noindent Without lost of generality we can suppose that $\delta(C,{\cal B})\leq \delta(D,{\cal B})$. 
Since $\delta(C,V)<\delta(C,{\cal B})$ so $\delta(C,V)<\delta(D,{\cal B})=\delta(D,V)$ and using $(\delta3)$ we get

\begin{equation}
\label{igualdad dos}
\delta(C,D)=\inf\{\delta(C,V),\delta(D,V)\}=\delta(C,V).
\end{equation}

\noindent From $(\ref{igualdad uno})$ and $(\ref{igualdad dos})$ it follows that
$\delta(C,D)=\delta(U,V)$. Moreover
$\delta(C,D)<\inf\{\delta(C,{\cal B}),\delta(D,{\cal B})\}$ and we are done.
\end{proof}

\begin{prop}
Let $C$ be a reduced curve with $r>1$ branches $C_i$ and let $D$ be a branch. Suppose that $\delta(C,D)<\hbox{\rm inf}\{\Delta, 
\hbox{\rm inf}_{i,j}
\{\delta(C_i,C_j)\}\}$, where $\Delta$ is a real number. Then 
$\delta(C_{i},D)<\Delta$, for $i\in \{1,\ldots,r\}$.
\end{prop}
\begin{proof} By definition we have $\delta(C,D)=\inf_{i=1}^{r}\{\delta(C_{i},D)\}$. Thus there exists $i_{0}\in \{1,\ldots,r\}$ such that $\delta(C,D)=\delta(C_{i_{0}},D)$. Fix $j_{0}\in \{1,\ldots,r\}$. By hypothesis $\delta(C_{i_{0}},D)<\delta(C_{i_{0}},C_{j_{0}})$ and after ($\delta_{3}$) we have $\delta(C_{i_{0}},D)=\delta(C_{j_{0}},D)<\delta(C_{i_{0}},C_{j_{0}})$. Now $\delta(C_{j_{0}},D)=\delta(C_{i_{0}},D)=\delta(C,D)<\Delta$ and we are done since $j_{0}\in \{1,\ldots,r\}$ is arbitrary.
\end{proof}

\begin{coro}
Let $C$ be a reduced curve with 
$r>1$ branches $C_i$ and let ${\cal B}$ be a family of 
branches such that the condition (*) holds. If $\delta(C,W)<\delta(C,{\cal B})$ for a branch $W\in {\cal B}$ then $\delta(C_{i},W)<\delta(C_{i},{\cal B})$, for $i\in \{1,\ldots,r\}$.
\end{coro}

\section{The contact exponent}
\label{sect3}
Recall that $d(C,D)=\frac{(C,D)_{0}}{m(C)m(D)}$ for any branches $C$ and $D$ (see Example \ref{ej} (1)).\\

If $C$ and $D$ are reduced curves with irreducible components $C_{i}$ and $D_{j}$ then we set $d(C,D)=\inf_{i,j}\{d(C_{i},D_{j})\}$.

\begin{lema}
\label{elemental d}
If $C$ has $r>1$ branches $C_i$ 
and $D$ is any branch then

\begin{enumerate}

\item $d(C,D)\leq \hbox{\rm inf}_{i,j}\{d(C_i,C_j)\}$,
\item $d(C,D)\leq \frac{(C,D)_0}{m(C)m(D)}$ with equality if
$d(C,D)< \hbox{\rm inf}_{i,j}\{d(C_i,C_j)\}.$
\end{enumerate}
\end{lema}

\begin{proof} 
The first part of the lemma follows from Lemma \ref{elemental}, for $\delta=d$.

In order to check the second part let us observe that 
\begin{eqnarray*}
(C,D)_0&=&\sum_{i=1}^r(C_i,D)_0=\sum_{i=1}^r d(C_i,D)m(C_i)m(D)\geq 
\sum_{i=1}^r d(C,D)m(C_i)m(D)\\
&=&d(C,D)m(C)m(D),
\end{eqnarray*}

so $d(C,D)\leq \frac{(C,D)_0}{m(C)m(D)}$ with equality if and only if 
$d(C,D)=d(C_i,D)$ for all $i\in\{1,\ldots,r\}$. 

Suppose that $d(C,D)<\hbox{\rm inf }_{i,j}
\{d(C_i,C_j)\}$. By definition
there is  $i_0\in\{1,\ldots,r\}$ such that $d(C,D)=d(C_{i_0},D)$, so $d(C_{i_0},D)<
d(C_{i_0},C_j)$ for all $j\in\{1,\ldots,r\}$. Applying $(\delta3)$ ($\delta=d$) to $C_{i_0}$, $D$ and 
$C_j$ we get 
$$d(C_j,D)=\hbox{\rm inf}\{d(C_{i_0},D),d(C_j,C_{i_0})\}=d(C_{i_0},D)=d(C,D)\;\;\; 
\hbox{\rm for all } j,$$

so $d(C,D)=\frac{(C,D)_0}{m(C)m(D)}$.
\end{proof}

\bigskip

Now we put for any reduced  curve  $C$:
\[d(C):=\hbox{\rm sup}\{d(C,W)\;:\;W\;\; \hbox{\rm runs over all smooth 
branches}\}\]

and call $d(C)$ the {\em contact exponent} of $C$ (see 
\cite[Definition 1.5]{Hironaka} where the term {\em characteristic exponent} 
is used). We say that a smooth germ $W$ has {\em maximal contact} with $C$ 
if $d(C,W)=d(C)$.\\

Observe that $d(C)=+\infty$ if $C$ is a smooth 
branch.\\

\begin{lema}
\label{maximal}
Let $C$ be  a singular branch with $\Gamma(C)=\langle v_{0},v_{1},\ldots, v_{g}\rangle$. Then there exists a smooth branch $W_{0}$ such that $(C,W_{0})_{0}=v_{1}$. Moreover, $d(C)=\frac{v_{1}}{v_{0}}$ and $W_{0}$ has maximal contact with $C$.
\end{lema}
\begin{proof}
See \cite[Proposition 3.6]{GB-P1} or \cite[Folgerung II.1.1]{Angermuller} for the first part of the lemma. To check the second part, let $W$ be a smooth branch. We have $d(C,W_{0})=\frac{v_{1}}{v_{0}}\not\in \N$ and $d(W,W_{0})=(W,W_{0})_{0}\in \N$. Therefore $d(C,W_{0})\neq d(W,W_{0})$ and $d(C,W)=\inf \{d(C,W_{0}), d(W,W_{0})\}\leq d(C,W_{0})$.
\end{proof}

\medskip

\begin{prop}
\label{teorema}
Let $C$ be a reduced curve with $r>1$ branches $C_i$. 
Then
\[d(C)=\hbox{\rm inf}\{\hbox{\rm inf}_i\{d(C_i)\}, \hbox{\rm inf}_{i,j}
\{d(C_i,C_j)\}\}.\]

Moreover, there exists $i_0\in\{1,\ldots,r\}$ such that if a 
smooth branch $W$ has maximal contact with the branch $C_{i_0}$ then 
it has maximal contact with the curve $C$.
\end{prop}

\begin{proof}
Use Theorem \ref{principal} when $\delta=d$ and ${\cal B}$ is the family of smooth branches.
\end{proof}

\begin{coro}
The contact exponent of a reduced curve is an equisingularity invariant.
\end{coro}
\begin{proof}
It is a consequence of Lemma \ref{maximal} and Proposition \ref{teorema}.
\end{proof}

\begin{coro}
\label{rrrr}
Let $C$ be a reduced curve  with $r\geq 1$ branches. Then $d(C)$ equals  $\infty$ or a rational number greater than or equal to 1. There exists a smooth curve $W$ that has maximal contact with $C$. Moreover,
\begin{enumerate}
\item $d(C)=+\infty$ if and only if $C$ is a smooth branch.
\item $d(C)=1$ if and only if $C$ has at least two tangents.
\item $d(C)<\inf_{i=1}^{r}\{d(C_{i})\}$ if and only if $d(C)$ is an integer.
\end{enumerate}
\end{coro}
\begin{proof}
The first and second properties follow  from  Lemma \ref{maximal} and Proposition \ref{teorema}.

To check the third part suppose that 
$d(C)\in \mathbb N$. Then $d(C)\neq \hbox{\rm inf}_i\{d(C_i)\}$ and 
by Proposition \ref{teorema} we get the inequality $d(C)<
\hbox{\rm inf}_i\{d(C_i)\}$.

\noindent Suppose now that $d(C)< \hbox{\rm inf}_i\{d(C_i)\}$. We have to check
that $d(C)\in \mathbb N$. By Proposition \ref{teorema} we get $d(C)=\hbox{\rm inf}_{i,j} \{d(C_i,C_j)\}=d(C_{i_0},C_{j_0})$ for some $i_0,j_0$.  By hypothesis 
$d(C)=d(C_{i_0},C_{j_0})<\hbox{\rm inf}\{d(C_{i_0}), d(C_{j_0})\}$. Hence by 
Proposition \ref{lema} ($\delta=d$) there is not a branch with maximal contact with 
$C_{i_0}$ and $C_{j_0}$ and $d(C)=d(C_{i_0},C_{j_0})=d(U,V)$ for some smooth branches $U,V$, and we conclude that $d(C)\in \N$.
\end{proof}

\begin{lema} 
\label{Noether}Let $C$ and $D$ be two branches with common tangent. Suppose that $m(C)=m(\widehat C)$ and $m(D)=m(\widehat D)$. Then 
\[d(C,D)=d(\widehat C,\widehat D)+1.\]
\end{lema}
\begin{proof} It is a consequence of Max Noether's theorem, which states $(C,D)_{0}=m(C)m(D)+(\widehat C, \widehat D)_{0}.$
\end{proof}

\begin{teorema}[Hironaka]
\label{th:Hironaka}
Let $\widehat C$  be the strict quadratic transformation of a reduced singular unitangent curve $C$. We get
\begin{enumerate}
\item[\hbox{\rm (i)}] if $d(C) < 2$ then $m(\widehat C)<m(C)$,
\item[\hbox{\rm (ii)}] if $d(C)\geq 2$ then $m(\widehat C)=m(C)$ and $d(\widehat C)=d(C)-1$,
\item[\hbox{\rm (iii)}] if $d(C)\geq 2$ and $W$ is a smooth curve tangent to $C$  then $d(C,W)=d(\widehat {C}, \widehat{W})+1$.  If $W$ has maximal contact with $C$ then $\widehat{W}$ has maximal contact with  $\widehat{C}$.
\end{enumerate}
\end{teorema}

\begin{proof}

Firstly consider the case when $C$ is a singular branch. Let $\Gamma(C)=\langle v_{0},v_{1}, \ldots, v_{g}\rangle$. Let us prove (i). By Lemma \ref{maximal} $d(C)=\frac{v_{1}}{v_{0}}$ so $d(C)<2$ if and only if $v_{1}-v_{0}<v_{0}$.  By de second part of Lemma \ref{prop:Angermuller} we have $m(\widehat C)=\min (\Gamma(\widehat C)\backslash \{0\})=v_{1}-v_{0}<v_{0}=\min (\Gamma(C)\backslash \{0\})=m(C).$\\
Now we will prove (ii) when $C$ is irreducible. Assume that $d(C)\geq 2$ (in fact $d(C)>2$ since $d(C)\not\in \N)$. The condition $d(C)\geq 2$ means $v_{0}<v_{1}-v_{0}$ and by the first part of Lemma \ref{prop:Angermuller} we get $\Gamma(\widehat C)=\langle v_{0},v_{1}-v_{0},\cdots \rangle.$ Consequently $m(\widehat C)=v_{0}=m(C)$ and $d(\widehat C)=\frac{v_{1}-v_{0}}{v_{0}}=d(C)-1$. \\

Now let $C=\bigcup_{i=1}^{r}C_{i}$, $r>1$ with irreducible $C_{i}$ and let us prove $(i)$ and $(ii)$ in this case.\\
Assume that $d(C)<2$. We claim that there exists $i_{0}\in \{1,\ldots,r\}$ such that $d(C)=d(C_{i_{0}}).$ Suppose that such $i_{0}$ does not exist. Then $d(C)\neq d(C_{i})$ for any $i\in \{1,\ldots,r\}$ and by Proposition \ref{teorema} $d(C)=\inf_{i,j}\{d(C_{i},C_{j})\}=d(C_{i_{0}},C_{j_{0}})$ for some $i_{0},j_{0}\in \{1,\ldots,r\}$. We claim that $d(C_{i_{0}})<2$ or $d(C_{j_{0}})<2$. In the contrary case, we had $d(C_{i_{0}})\geq 2$ and $d(C_{j_{0}})\geq 2$ and we would get $m(\widehat C_{i_{0}})=m(C_{i_{0}})$ and $m(\widehat C_{j_{0}})=m(C_{j_{0}}),$ which implies by Lemma \ref{Noether} $d(C_{i_{0}},C_{j_{0}})=d(\widehat C_{i_{0}},\widehat C_{j_{0}})+1\geq 2$. This is a contradiction since $d(C_{i_{0}},C_{j_{0}})=d(C)<2$.\\
If $d(C_{i_{0}})=d(C)<2$ then by the irreducibility case, $m(\widehat C_{i_{0}})<
m(C_{i_{0}})$ and $m(C)-m(\widehat C)=\sum_{i=1}^{r}(m(C_{i})-m(\widehat C_{i}))\geq m(C_{i_{0}})-m(\widehat C_{i_{0}})>0$.\\

Suppose now that $d(C)\geq 2$. We have 
\[\inf\{d(C_{i})\}\geq \inf\{\inf (d(C_{i})),\inf (d(C_{i},C_{j}))\}=d(C)\geq 2.\]
 
Thus $d(C_{i})\geq 2$ for $i\in \{1,\ldots,r\}$ and by the first part of the proof $m(\widehat C_{i})=m(C_{i})$ and $d(\widehat C_{i})=d(C_{i})-1.$ 
Hence $m(\widehat C)=m(C)$.
Moreover, by Lemma \ref{Noether}, $d(C_{i},C_{j})=d(\widehat C_{i},\widehat C_{j})+1$ 
and $d(C)= \inf\{\inf (d(C_{i})),\inf (d(C_{i},C_{j}))\}=\inf\{\inf (d(\widehat C_{i})), \inf (d(\widehat C_{i},\widehat C_{j}))\}+1=d(\widehat C)+1.$\\

To finish let us prove (iii). By Lemma \ref{Noether} $d(C_{i},W)= d(\widehat C_{i},\widehat W)+1$ for $i\in \{1,\ldots,r\}$ and $d(C,W)=\inf \{d(C_{i},W)\}=\inf \{d(\widehat C_{i},W)\}+1=d(\widehat C,W)+1$. Suppose that $W$ has maximal contact with $C$.
Then $d(C)=d(C,W)=d(\widehat C,\widehat W)+1\leq d(\widehat C)+1=d(C),$ where the last equality is a consequence of statement (ii) of the theorem. This implies $d(\widehat C,\widehat W)=d(\widehat C)$. Thus $\widehat W$ has maximal contact with $\widehat C$.
\end{proof}

\medskip 

\begin{lema}
\label{444}
Let $C$ be a reduced curve  with $r> 1$ branches and $W$ a smooth branch.
If $d(C,W)\not\in \mathbb N$ then $d(C,W)=d(C)$. 
\end{lema}
\begin{proof}
The lemma is obvious if $C$ is a branch. In the general case $d(C,W)=\hbox{\rm inf}_i\{d(C_i,W)\}=d(C_{i_0},W)$ for
some $i_0\in \{1,\ldots,r\}$. If $d(C,W)\not\in \mathbb N$ then $d(C_{i_0},W)\not\in \mathbb N$ and $d(C_{i_0},W)=d(C_{i_0})$ since $C_{i_0}$ is a branch. Consequently, we get $d(C,W)=d(C_{i_0})$ which implies, by Proposition \ref{teorema}, $d(C)=d(C_{i_0})=d(C,W)$.
\end{proof}

\medskip

Now we give a characterization of smooth curves which does not have maximal contact with a reduced curve.

\medskip
\begin{prop}
Let $C$ be a reduced curve  with $r> 1$ branches.  A 
smooth branch $W$ does not have maximal contact with $C$ 
if and only if $(C,W)_0< d(C)m(C)$. Moreover, in this case 
$(C,W)_0\equiv 0$ \hbox{\rm (mod} $m(C)$\hbox{\rm )}.
\end{prop}
\begin{proof}
Let us suppose that $W$ is a smooth branch which does not have 
maximal contact with $C$. We will check that $(C,W)_0<d(C)m(C)$ and
$\frac{(C,W)_0}{m(C)}\in \mathbb N$. By Proposition \ref{teorema} we get
$d(C,W)<\hbox{\rm inf }_{i,j}\{d(C_i,C_j)\}$ since $d(C,W)<d(C)$.
According to the second part of Lemma \ref{elemental d} we can write
$d(C,W)=\frac{(C,W)_0}{m(C)}$, thus $(C,W)_0=d(C,W)m(C)<d(C)m(C)$.
We claim
$\frac{(C,W)_0}{m(C)}=d(C,W)$ is an integer. Indeed, by  Lemma \ref{444} we get
$d(C,W)=d(C)$, which is a contradiction.

\noindent Now suppose that $(C,W)_0<d(C)m(C)$. By the second part of
Lemma \ref{elemental d} we get $d(C,W)\leq \frac{(C,W)_0}{m(C)}$ and consequently $d(C,W)<d(C)$, which means that $W$ does not have maximal contact with $C$.
\end{proof}

\section{Milnor number and Hironaka contact exponent}
\label{sect4}
Let $C$ be a reduced curve. We define the {\em Milnor number} $\mu(C)$ of $C$ by the formula $\mu(C)=c(C)-r(C)+1$, where $c(C)$ is the degree of the conductor of the local ring of $C$ and $r(C)$ is the number of branches (see Preliminaries).\\

If $C:\{f=0\}$ then $\mu(C)=\dim_{K}K[[x,y]]/\left(\frac{\partial f}{\partial x},\frac{\partial f}{\partial y}\right)$ provided that $K$ is of characteristic zero (see \cite{GB-P2}).

\begin{lema}
\label{Milnor}
Let $C=\bigcup_{i=1}^{r}C_{i}$, where $r\geq 1$ and $C_{i}$ are irreducible. Then
\begin{enumerate}
\item $\mu(C)+r-1=\sum_{i=1}^{r}\mu(C_{i})+2\sum_{1\leq i<j \leq r}(C_{i},C_{j})_{0}$,
\item if $C$ is a branch then $\mu(C)$ equals the conductor of the semigroup $\Gamma(C)$,
\item $\mu(C)\geq 0$ with equality if and only if $C$ is a smooth branch.
\end{enumerate}
\end{lema}
\begin{proof}
See \cite[Proposition 2.1]{GB-P2}.
\end{proof}

\medskip

\begin{prop}
\label{Milnor}
Let $C=\bigcup_{i=1}^rC_i$ be a singular reduced curve with 
$r$ branches $C_i$. Then $\mu(C)\geq (d(C)m(C)-1)(m(C)-1)$ with equality if and only if the following two conditions are satisfied:

\begin{enumerate}
\item [$(e_1)$] $d(C_i,C_j)=d(C)$ for all $i\neq j$,

\item [$(e_2)$] if the branch $C_i$ is singular then $C_i$ has 
exactly one Zariski pair and $d(C_i)=d(C)$.
\end{enumerate}
\end{prop}

\begin{proof}
First let us suppose that $C$ is a branch with 
$\Gamma(f)=\langle v_0,v_1,\ldots,v_g\rangle$. Let $n_{0}=1$. Since
$n_{i-1}v_{i-1}\leq v_{i}$ for $i\in \{1,\ldots,g\}$ we have $n_{0}n_{1}\cdots n_{k-1}v_{1}\leq v_{k}$ for $k\in \{1,\ldots,g\}$. We get
\begin{eqnarray*}
c(C)&=&\sum_{k=1}^{g}(n_{k}-1)v_{k}-v_{0}+1\geq \sum_{k=1}^{g}((n_{k}-1)n_{k-1}\cdots n_{1}n_{0})v_{1}-v_{0}+1\\
&=&(n_{g}\cdots n_{1}n_{0}-n_{0})v_{1}-(v_{0}-1)=(v_{0}-1)v_{1}-(v_{0}-1)\\
&=&(v_{0}-1)(v_{1}-1).
\end{eqnarray*}
Moreover, $c(C)=(v_{0}-1)(v_{1}-1)$ if and only if $\Gamma(C)=\langle v_{0},v_{1}\rangle$.\\

Now suppose that the curve $C$ has $r>1$ branches $C_i$ and let $\overline  m_i=m(C_i)$ for $i\in\{1,\ldots,r\}$. From Proposition \ref{teorema} 
we get $d(C_i)\geq d(C)$ and $d(C_i,C_j)\geq d(C)$ for all $i,j\in \{1,
\ldots,r\}$. By the first part of the proof $\mu(C_i)\geq 
(d(C_i)\overline m_i-1)(\overline m_i-1)$ for the singular branches with equality if and only if $C_i$ is a singular branch satisfying condition $(e_2)$.

Let $I:=\{i\;:\;C_i \;\hbox{\rm is singular}\}$.
\medskip
Now we get 
\begin{eqnarray*}
\mu(C)+r-1 & = & \sum_{i=1}^r \mu(C_i) +2\sum_{1\leq i<j\leq r}(C_i,C_j)_0 \\
 & = & \sum_{i=1}^r \mu(C_i) +2\sum_{1\leq i<j\leq r}d(C_i,C_j)\overline  m_i\overline  m_j \\
 & \geq & \sum_{i\in I} (d(C_i)\overline  m_i-1)(\overline  m_i-1)+2\sum_{1\leq i < j \leq r} 
d(C_i,C_j)\overline  m_i\overline  m_j\\
 & \geq & \sum_{i=1}^r (d(C)\overline  m_i-1)(\overline  m_i-1)+2\sum_{1\leq i < j \leq r} 
d(C)\overline  m_i\overline  m_j\\
 & = & d(C)(m(C)^2-m(C))-m(C)+r
\end{eqnarray*}

with equality if and only if the conditions $(e_1)$ and $(e_2)$ 
are satisfied.
\end{proof}

\begin{lema}
\label{una sola tangente}
Let $C$ be a unitangent singular curve. We have:
\begin{enumerate}

\item $d(C)\geq 1+\frac{1}{m(C)}$. Moreover $d(C)=1+\frac{1}{m(C)}$ if and only if $C$ is a branch of semigroup $\langle m(C),m(C)+1\rangle$.

\item $\mu(C)\geq m(C)(m(C)-1)$ with equality if and only if $d(C)=1+\frac{1}{m(C)}$.
\end{enumerate}
\end{lema}

\begin{proof}
Let $\{C_i\}_i$ be the set of branches of $C$. To check the first part of the lemma we may assume that $d(C)$ is not an integer. Then by
Proposition \ref{teorema} and the third part of Corollary \ref{rrrr} 
there is an $i_0$ such that $d(C)=d(C_{i_0})$. The contact exponent 
$d(C_{i_0})$ is a fraction with the denominator less than or equal to 
$m(C_{i_0})$. Therefore we get $d(C)=d(C_{i_0})\geq 1+\frac{1}{m(C_{i_0})}
\geq 1+\frac{1}{m(C)}$ and the equality $d(C)=1+\frac{1}{m(C)}$ implies 
$m(C_{i_0})=m(C)$ and consequently $C_{i_0}=C$. Moreover the semigroup of $C$ is $\langle m(C),m(C)+1\rangle$ since $m(C)$ and $m(C)+1$ are coprime.

\medskip

\noindent In order to prove the second part we get, by Proposition \ref{Milnor} and the first part of this lemma,
\begin{eqnarray*}
\mu(C)& \geq &(d(C)m(C)-1)(m(C)-1)\\
 & \geq &\left(\left(1+\frac{1}{m(C)}\right)m(C)-1\right)(m(C)-1)=m(C)(m(C)-1).
\end{eqnarray*}

\noindent If $\mu(C)=m(C)(m(C)-1)$ then from the above calculation it follows that $d(C)=1+\frac{1}{m(C)}$.

\medskip 

\noindent On the other hand if $d(C)=1+\frac{1}{m(C)}$ then by the first part of this lemma $C$ is a branch of semigroup $\langle m(C),m(C)+1\rangle$. According to Proposition \ref{Milnor} $\mu(C)=(d(C)m(C)-1)(m(C)-1)=m(C)(m(C)-1)$.
\end{proof}

\medskip 

If $\mu(C)=(d(C)m(C)-1)(m(C)-1)$ then the pair $(m(C), d(C))$ determines the equisingularity class of $C$. More specifically, we have:

\begin{prop}
\label{Eggers2}
Let $C$ be a reduced  singular curve. Then $\mu(C)=(d(C)m(C)-1)(m(C)-1)$ if and only if one of the following three conditions holds

\begin{enumerate}
\item[{\rm (1)}] $d(C)\in \mathbb N$. All branches of $C$ are smooth and intersect pairwise with multiplicity $d(C)$.

\item[{\rm (2)}]  $d(C)\not\in \mathbb N$ and $m(C)d(C)\in \mathbb N$. The curve $C$ has  $r=\gcd(m(C),m(C)d(C))$ branches, each with semigroup generated by $\left(\frac{m(C)}{r},\frac{m(C)d(C)}{r}\right)$, intersecting pairwise with multiplicity $\frac{m(C)^2d(C)}{r^2}$.

\item[{\rm (3)}]  $m(C)d(C)\not\in \mathbb N$. There is a smooth curve $L$ such that $C=L\cup C'$, where $C'$ is a curve of type $(2)$ with $d(C')=d(C)$ and $m(C')=m(C)-1$. The branch $L$ has maximal contact with any branch of $C'$.
\end{enumerate}
\end{prop}

\begin{proof}
 If one of conditions $(1),(2)$ or $(3)$ is satisfied then a direct calculation shows that $\mu(C)=(d(C)m(C)-1)(m(C)-1)$.

Suppose that $C=\bigcup_{i=1}^rC_i$ satisfy the equality $\mu(C)=(d(C)m(C)-1)(m(C)-1)$.
By Proposition \ref{Milnor} the conditions $(e_1)$ and $(e_2)$ are satisfied. Let us consider three cases:

\medskip

Case 1: All branches $C_i$ are smooth. Then $C$ is of type $(1)$ by $(e_1)$.

\medskip

Case 2: All branches $C_i$ are singular. Then the branches $\{C_i\}_i$ have the same semigroup $\langle v_0,v_1 \rangle$ and according to $(e_2)$ $d(C_i)=d(C)$ for all $i\in\{1,\ldots,r\}$. Clearly, we have $m(C)=\sum_{i=1}^rm(C_i)=rv_0$ and $m(C)d(C)=m(C)d(C_i)=rv_1$. Thus $m(C)d(C)\in \mathbb N$, $r=\gcd(m(C),m(C)d(C))$ and it is easy to see that $C$ is of type $(2)$.

\medskip

Case 3: Neither Case $1$ nor Case $2$ holds, thus $r>1$. We may assume that $C_1$ is smooth and $C_2$ is singular. If $r>2$ then all branches $C_i$ for $i\geq 3$ are singular. In fact, we have by $(e_1)$: $d(C_1,C_i)=d(C_2,C_i)=d(C_1,C_2)=d(C)$ and by $(e_2)$: $d(C)=d(C_2)\not\in \mathbb N$. Thus $d(C_1,C_i)\not\in \mathbb N$ and $C_i$ are singular for all $i\geq 3$. Let $L:=C_1$ and $C':=\bigcup_{i=2}^r(C_i,0)$. Then $C=L\cup C'$ and we check using Proposition \ref{teorema} that $C$ is of type $(3)$.
\end{proof}

\begin{coro}
\label{c:equisingular}
Let $C_{1},C_{2}$ be two reduced singular curves such that $\mu(C_{i})=(d(C_{i})m(C_{i})-1)(m(C_{i})-1)$ for $i\in \{1,2\}$. Then $C_{1}$ and $C_{2}$ are equisingular if and only if $(m(C_{1}),d(C_{1}))=(m(C_{2}),d(C_{2}))$.
\end{coro}

\begin{coro}
\label{entero}
Let $C$ be a reduced singular curve. Suppose that $\mu(C)=(d(C)m(C)-1)(m(C)-1)$ and $m(C)d(C)\not\in \N$. Then $(m(C)-1)d(C)\in \N$.
\end{coro}

To compute $\mu(C)$ one can use Pham's formula.\\

\begin{prop}[\cite{Pham}]
Let $C=\bigcup_{i}^{t}C_{i}$, where $C_{i}$ are unitangent and the tangents to $C_{i}$ and $C_{j}$ are different for $i\neq j$. Then 
\[\mu(C)+t(C)-1=m(C)(m(C)-1)+\sum_{k=1}^t \mu(\widehat{C}_k).\]
\end{prop}

\begin{proof}
We distinguish three cases.\\

Suppose that $C$ is irreducible. Then $\mu(C)=m(C)(m(C)-1)+\mu(\widehat{C})$ by the well-known formula $c(C)=m(C)(m(C)-1)+c(\widehat{C})$ (see \cite[Korollar II.1.8]{Angermuller}).\\

Suppose now that $C$ is unitangent and let $C=\bigcup_{i}^{r}C_{i}$, where $C_{i}$ are irreducible and let $\overline  m_i=m(C_i)$ for $i\in\{1,\ldots,r\}$. Then
\begin{eqnarray*}
\mu(C)+r-1 & = & \sum_{i=1}^r \mu(C_i) +2\sum_{1\leq i<j\leq r}(C_i,C_j)_0 \\
 & = & \sum_{i=1}^t \left(\overline  m_i (\overline  m_i-1)+\mu(\widehat{C_{i}})\right)+2\sum_{1\leq i<j\leq r}\left(\overline  m_i\overline  m_j+(\widehat{C_{i}},\widehat{C_{j}})_{0} \right)\\
 & = &  \sum_{i=1}^r\overline  m_i (\overline  m_i-1)+2\sum_{1\leq i<j\leq r}\overline  m_i\overline  m_j+\sum_{i=1}^r\mu(\widehat{C_{i}})+2\sum_{1\leq i<j\leq r}(\widehat{C_{i}},\widehat{C_{j}})_{0}\\
  & = &  \sum_{i=1}^r\overline  m_i (\overline  m_i-1)+2\sum_{1\leq i<j\leq r}\overline  m_i\overline  m_j+\mu(\cup_{i=1}^{r}\widehat{C_{i}})+r-1\\
  & = &  \sum_{i=1}^r\overline  m_i (\overline  m_i-1)+2\sum_{1\leq i<j\leq r}\overline  m_i\overline  m_j+\mu(\widehat{\cup_{i=1}^{r}{C_{i}}})\\
 & = & m(C)(m(C)-1)+\mu(\widehat{C})+r-1.
\end{eqnarray*}

Finally suppose that $C=\bigcup_{i}^{t}C_{i}$, where $C_{i}$ are unitangent and the tangents to $C_{i}$ and $C_{j}$ are different for $i\neq j$. Put $\overline  m_i=m(C_i)$ for $i\in\{1,\ldots,t\}$. Then
\begin{eqnarray*}
\mu(C)+t-1 & = & \sum_{i=1}^t \mu(C_i) +2\sum_{1\leq i<j\leq t}(C_i,C_j)_0 \\
 & = & \sum_{i=1}^t \left(\overline  m_i (\overline  m_i-1)+\mu(\widehat{C_{i}})\right)+2\sum_{1\leq i<j\leq t}\overline  m_i\overline  m_j\\
 & = &  \sum_{i=1}^t\mu(\widehat{C_{i}})+ \sum_{i=1}^t(\overline m_{i})^{2}+2\sum_{1\leq i<j\leq t}\overline  m_i\overline  m_j-\sum_{i=1}^t\overline m_{i}\\
 & = & m(C)(m(C)-1)+ \sum_{i=1}^t\mu(\widehat{C_{i}}).
\end{eqnarray*}
\end{proof}

\section{Contact exponents of higher order}
\label{sect5}
Let ${\cal B}_k$ be the family of branches having at most  $k-1$ Zariski pairs. If $C$ is a reduced curve we put 
\[d_k(C):=\hbox{\rm sup}\{d(C,W)\;:\;W\in {\cal B}_k\}=d(C,{\cal B}_k).\]

Observe that  $d_1(C)=d(C)$.

A  branch $D\in {\cal B}_k$ has $k$-maximal contact with $C$ if  $d(C,D)=d_k(C)$.\\

The concept of contact exponent of higher order was studied by Lejeune-Jalabert \cite{LJ} and Campillo \cite{Campillo-libro}.

\begin{lema}
\label{lemak1}
Let $C:\{f=0\}$ be  a singular branch with $\Gamma(C)=\langle v_{0},v_{1},\ldots, v_{g}\rangle$.\\There exist irreducible power series $f_{0},\ldots,f_{g-1}$ such that $\ord f_{k-1}=\frac{v_{0}}{e_{k-1}}$ and $(f,f_{k-1})_{0}=v_{k}$.
\end{lema}
\begin{proof} 
We may assume that $(f,x)_{0}=\ord f$. According to \cite[Theorem 3.2]{GB-P1} there exist distinguished polynomials 
$f_{0},\ldots,f_{g-1}$ such that $(f_{k-1},x)_{0}=\frac{v_{0}}{e_{k-1}}$ and $(f,f_{k-1})_{0}=v_{k}$. Consider the sequence $d(f,x)=1$, $d(f_{k-1},x)=\frac{(f_{k-1},x)_{0}}{\ord f_{k-1}}$ and $d(f,f_{k})=\frac{v_{k}}{v_{0}\frac{v_{0}}{e_{k-1}}}=\frac{e_{k-1}v_{k}}{(v_{0})^{2}}$. Since $d(f_{k-1},x)=\frac{e_{k-1}v_{k}}{(v_{0})^{2}}\geq \frac{e_{0}v_{1}}{(v_{0})^{2}}=\frac{v_{1}}{v_{0}}>1$ we have $d(f_{k-1},x)=d(f,x)=1$, that is $(f_{k-1},x)_{0}=\ord f_{k-1}$.
\end{proof}

\begin{lema}
\label{lemak2} Let $C:\{f=0\}$ be  a singular branch with $\Gamma(C)=\langle v_{0},v_{1},\ldots, v_{g}\rangle$.
If $E$ is a branch such that $d(C,E)>\frac{e_{k-1}v_{k}}{(v_{0})^{2}}$ then $E$ has at least $k$ Zariski pairs.
\end{lema}
\begin{proof} See \cite[Theorem 5.2]{GB-P1}.
\end{proof}
\begin{prop}
Let $C$ be a branch with $\Gamma(C)=\langle v_0,\ldots,v_{g}\rangle$. Then $d_k(C)=\frac{e_{k-1}v_{k}}{(v_0)^2}$, where $e_{k-1}=\gcd(v_0,\ldots,v_{k-1})$.
\end{prop}

\begin{proof}
By Lemma \ref{lemak1} there is $D_{k-1}\in {\cal B}_k$ such that $\ord D_{k-1}=\frac{v_{0}}{e_{k-1}}$ and $(C,D_{k-1})_{0}=v_{k}$. Then $d_{k}(C)\geq d(C,D_{k})=\frac{(C,D_{k-1})_{0}}{\ord C \ord D_{k-1}}=\frac{e_{k-1}v_{k}}{(v_{0})^{2}}$. Suppose now that there is a branch $E\in {\cal B}_k$ such that $d(C,E)>\frac{e_{k-1}v_{k}}{(v_{0})^{2}}$. Then $\frac{(C,E)_{0}}{v_{0}\ord E}>\frac{e_{k-1}v_{k}}{(v_{0})^{2}}$, hence $\frac{(C,E)_{0}}{\ord E}>\frac{e_{k-1}v_{k}}{v_{0}}$. By Lemma \ref{lemak2} we conclude that $E$ has at least $k$ Zariski pairs which is a contradiction (since  $E\in {\cal B}_k$).
\end{proof}

\begin{prop}
Let $C$ be a reduced curve with $r>1$ branches $C_i$. 
Then
\[d_{k}(C)=\hbox{\rm inf}\{\hbox{\rm inf}_i\{d_{k}(C_i)\}, \hbox{\rm inf}_{i,j}
\{d(C_i,C_j)\}\}.\]
\end{prop}

\begin{proof}
Use Theorem \ref{principal} when $\delta=d$ and ${\cal B}_{k}$ is the family of  branches having at most  $k-1$ Zariski pairs.
\end{proof}

\section{Polar invariants and the contact exponent}
\label{sect6}
Let $K$ be a field of characteristic zero. Let $C$ be a reduced plane singular curve and let $P(C)$ be a generic polar of $C$. Then $P(C)$ is a reduced germ of multiplicity $m(P(C))= m(C)-1$. Let $P(C)=\bigcup_{j=1}^sD_j$ be the decomposition of $P(C)$ into branches $D_j$.

\medskip
\noindent We put ${\cal Q}(C)=\left\{\frac{(C,D_j)_0}{m(D_j)}\;:\;j\in\{1,\ldots,
s\right\}$ and call the elements of ${\cal Q}(C)$ the {\em polar invariants } of 
$C$. They are equisingularity invariants of $C$ (see \cite{Teissier}, \cite{Gw-P}). 
In particular if $C$ is a branch then 

\[Q(C):=\{m(C)d_k(C)\}_{k=1}^{g}.\]

\medskip

Let us consider the minimal polar invariant 
$\alpha(C):=\hbox{\rm inf }{\cal Q}(C)$. 

\begin{prop}
\label{Lenarcik-Masternak-Ploski}
For any singular reduced germ $C$ we have $\alpha(C)=m(C)d(C)$.
\end{prop}  

\noindent {\bf Proof.-} See \cite[Theorem 2.1 (iii)]{L-M-P}.

\medskip
\noindent One could prove Proposition \ref{Lenarcik-Masternak-Ploski} by 
using Theorem \ref{teorema} and the explicit formulae for the polar invariants given in \cite[Theorem 1.3]{Gw-P}.

\medskip
\noindent We say that $C$ is an {\em Eggers singularity} if ${\cal Q}(C)$ has exactly one element.

\medskip
\begin{prop}
\label{Eggers}
Let $C$ be a singular reduced curve. Then $\mu(C)=(d(C)m(C)-1)(m(C)-1)$ if and only
if $C$ is an Eggers singularity.
\end{prop}

\begin{proof}

\noindent By \cite{Teissier2} Proposition 1.2 we get
\begin{eqnarray*}
\mu(C) &= &(C,P(C))_0-m(C)+1=\sum_{j=1}^s(C,D_j)_0-m(C)+1\\
 & \geq & \alpha(C)m(P(C))-m(C)+1=\alpha(C)(m(C)-1)-m(C)+1\\
 & = & (\alpha(C)-1)(m(C)-1)
\end{eqnarray*}

\noindent with equality if and only if $C$ is an Eggers singularity. We 
use Proposition \ref{Lenarcik-Masternak-Ploski}.

\end{proof}

\medskip

Proposition \ref{Eggers2} provides an explicit description of Eggers singularities.

\begin{coro}(\cite[p. 16]{Eggers})
If $C$ has exactly one polar invariant then $C$ is equisingular to $y^{n}-x^{m}=0$ or 
$y^{n}-yx^{m}=0$, for some integers $1<n<m$.
\end{coro}

\begin{proof}
We check that if $C:\{y^{n}-x^{m}=0\}$ then $m(C)=n$, $d(C)=\frac{m}{n}$ and $\mu(C)=nm-n-m+1$. On the other hand if $C:\{y^{n}-yx^{m}=0\}$ then $m(C)=n$, $d(C)=\frac{m}{n-1}$ and $\mu(C)=nm-n+1$. In both cases $\mu(C)=(d(C)m(C)-1)(m(C)-1)$, that is $C$ is an Eggers singularity.\\

Now let $C$ be an Eggers singularity. If $m(C)d(C)\in \N$ then $C$ and $\{y^{m(C)}-x^{m(C)d(C)}=0\}$ are equisingular by Corollary \ref{c:equisingular}. Analogously, if 
$m(C)d(C)\not\in \N$ then, by Corollary \ref{entero}, $(m(C)-1)d(C)\in \N$ and $C$ is 
equisingular to $\{y^{m(C)}-yx^{(m(C)-1)d(C)}=0\}$.
\end{proof}

\noindent {\small  Evelia Rosa Garc\'{\i}a Barroso\\
Departamento de Matem\'aticas, Estad\'{\i}stica e I.O. \\
Secci\'on de Matem\'aticas, Universidad de La Laguna\\
Apartado de Correos 456\\
38200 La Laguna, Tenerife, Espa\~na\\
e-mail: ergarcia@ull.es}

\medskip

\noindent {\small Arkadiusz P\l oski\\
Department of Mathematics and Physics\\
Kielce University of Technology\\
Al. 1000 L PP7\\
25-314 Kielce, Poland\\
e-mail: matap@tu.kielce.pl}


\begin{thebibliography}{GGGGG}

\bibitem[Ang]{Angermuller} Angerm\"uller, G. Die Wertehalbgruppe einer ebener
irreduziblen algebroiden Kurve. Math. Z. {\bf 153} (1977), no. 3, 267-282.

\bibitem[B-K] {Brieskorn-Knorrer} Brieskorn, E. and H. Kn\" orrer. {\em Plane 
algebraic curves}, Birkh\"auser Verlag 1986.

\bibitem[C]{Campillo-libro} Campillo, A. Algebroid curves in positive
characteristic.  Lecture Notes in Mathematics, 813. Springer Verlag, Berlin,
1980. v+168 pp.

\bibitem[E] {Eggers} Eggers, H. {\em Polarinvarianten und die Topologie von
Kurvensingularit\" aten}, Bonner Mathematische Schriften, Vol. 147, 1983.

\bibitem[GB-L-P]{GB-L-P} Garc\'{\i}a Barroso, E., Lenarcik, A. and A. P\l oski. {\em Newton diagrams and equivalence of plane curve germs.} J. Math. Soc. Japan {\bf 59}, no. 1 (2007), 81-96.

\bibitem[GB-P1]{GB-P1} Garc\'{\i}a Barroso, E. and A. P\l oski. {\em An approach to plane algebroid branches.} Rev. Mat. Complut., 28 (1) (2015), 227-252.

\bibitem[GB-P2]{GB-P2} Garc\'{\i}a Barroso, E. and A. P\l oski. {\em On the Milnor Formula in arbitrary characteristic.} Singularities, Algebraic Geometry, Commutative Algebra and Related Topics. Festschrift for Antonio Campillo on the occasion of his 65th Birthday. G.M. Greuel, L. Narva\'ez and S. Xamb\'o-Descamps eds. Springer, (2018), 119-133. 

\bibitem[Gw-P] {Gw-P} Gwo\'zdziewicz, J. and A. P\l oski, {\em On the 
polar quotients of an analytic plane curve}, Kodai Math. Journal, Vol. 25, 
1, (2002), 43-53. 
(1995) 199-210.
 
\bibitem[Hi] {Hironaka} Hironaka, H. {\em Introduction to the theory of 
infinitely near singular points}, Memorias del Instituto Jorge Juan 28, 
Madrid 1974.

\bibitem[M-W]{Melle-Wall} Melle-Hern\'andez, A. and C.T. C. Wall. {\em Pencils of curves on smooth surfaces}, Proc. Lond. Math. Soc., III Ser. 83 (2), 2001, 257-278.

\bibitem[LJ] {LJ} Lejeune-Jalabert, M. Sur l'\'{e}quivalence des courbes
alg\'ebro\"{\i}des planes. Coefficients de Newton. Contribution \`{a}
l'etude des singularit\'{e}s du poit du vue du polygone de Newton,
Paris VII, Janvier 1973, Th\`{e}se d'Etat.\\
See also in Travaux en Cours, 36 (edit. L\^{e} D\~{u}ng Tr\~{a}ng)
{\em Introduction \`{a} la th\'{e}orie des singularit\'{e}s I\/}, 49-124, 1988.

\bibitem[L-M-P] {L-M-P} A. Lenarcik, M. Masternak, A. 
P\l oski, {\em Factorization of the polar curve and the Newton polygon},
Kodai Math. J. 26 (2003), no. 3, 288-303.

\bibitem[Ph]{Pham} Pham, F. {\em Courbes discriminantes des singularit\'es planes d'ordre 3}, Singularit\'es \`a Carg\`ese 1972, Asterisque 7-8 (1973), 363-391.

\bibitem[T1] {Teissier2} B. Teissier, {\em Cycles \'evanescents, sections 
planes et condition de Whitney}, Ast\'erisque 7-8, (1973) 285-362.

\bibitem[T2] {Teissier} Teissier, B. {\em Complex curve singularities: a biased introduction.} Singularities in geometry and topology, 825Ð887, World Sci. Publ., Hackensack, NJ, 2007.

\bibitem[W]{Waldi} Waldi, R. {\em On the equivalence of plane curve singularities.} Communications in Algebra, 28(9), (2000), 4389-4401.

\end{thebibliography}
\end{document}